\documentclass{article}

\usepackage{amsfonts,amsmath,amsthm}
\usepackage{amssymb,latexsym,epsfig,color}
\usepackage{enumerate}
\usepackage[all,cmtip]{xy}
\usepackage{cite}
\usepackage[margin={2cm, 2cm, 2cm, 2cm}]{geometry}

\newcommand{\N}{{\mathbb N}}
\newcommand{\R}{{\mathbb R}}

\newcommand{\A}{{\mathcal A}}

\newcommand{\Set}{{\mathcal S}}

\newcommand{\C}{{\mathcal C}}

\newcommand{\NN}{{\mathcal N}}
\newcommand{\PP}{{\mathcal P}}
\newcommand{\TT}{{\mathcal T}}
\newcommand{\G}{{\mathcal G}}
\newcommand{\U}{{\mathcal U}}
\newcommand{\D}{{\mathcal D}}
\newcommand{\II}{{\mathcal I}}
\newcommand{\F}{{\mathcal F}}
\newcommand{\HH}{{\mathcal H}}
\newcommand{\px}{{2^{\PP(X)}}}

\newcommand{\upar}{{\uparrow}}
\newcommand{\dona}{{\downarrow}}
\theoremstyle{definition}
\newtheorem{theorem}{Theorem}[section]

\newtheorem{corollary}[theorem]{Corollary}
\newtheorem{definition}[theorem]{Definition}

\newtheorem{lemma}[theorem]{Lemma}
\newtheorem{example}[theorem]{Example}

\newtheorem{remark}[theorem]{Remark}

\def\({\langle}

\def\){\rangle}

\begin{document}

\title{Topologies as points within a Stone space:  lattice theory meets topology.}
\author{
Jorge L. Bruno$^a$
\\ \texttt{brujo.email@gmail.com}\\
\and
Aisling E. McCluskey$^{b,}$\footnote{Corresponding author.}  \\ \texttt{aisling.mccluskey@NUIgalway.ie}\\
} %end author

%\subjclass[2010]{Primary 06B23, 54H10; 06E15}
%\keywords{ Compactness;completeness; compactifications in $Top(X)$;
% Borel sets}
\date{$a$,$b$: National University of Ireland, Galway\\
School of Mathematics, NUI Galway, Ireland \\
{\bf{t}}: 00353 91 493162; {\bf{f}}: 00353 91 494542.
}

\maketitle

\begin{abstract}
For a non-empty set $X$, the collection $Top(X)$ of all topologies on $X$ sits inside the Boolean lattice $\PP(\PP(X))$ (when ordered by set-theoretic inclusion) which in turn can be naturally identified with the Stone space $\px$. Via this identification then, $Top(X)$ naturally inherits the subspace topology from $\px$. Extending ideas of Frink (1942), we apply lattice-theoretic methods to establish an equivalence between the topological closures of sublattices of $\px$ and their (completely distributive) completions. We exploit this equivalence when searching for countably infinite compact subsets within $Top(X)$ and in crystalizing the Borel complexity of $Top(X)$. We exhibit infinite compact subsets of $Top(X)$ including, in particular, copies of the Stone-\v{C}ech and one-point compactifications of discrete spaces.

\end{abstract}
\hspace{.8cm} 2010 Mathematics Subject Classification: Primary  06B23, 54H10; Secondary 06E15.

\maketitle

\section{Introduction}

For a non-empty set $X$, the collection $Top(X)$ of all topologies on $X$ sits inside the Boolean lattice $\PP(\PP(X))$ (when ordered by set-theoretic inclusion) which in turn can be naturally identified with the Stone space $\px$. Via this identification then, $Top(X)$ naturally inherits the subspace topology from $\px$ (see \cite{TopX1}), a subspace about which little is known. Frink  \cite{MR0006496} showed that endowing the lattice $\PP(\PP(X))$ with either the \textit{interval} or the \textit{order} topology yields the same space as $\px$. In the same paper, Frink also proved that a lattice is complete if and only if it is compact in its interval topology.

These ideas enable us to apply lattice-theoretic techniques in the investigation of an object, whose individual elements provide a rich source of topological inquiry and knowledge. Just as the Stone-\v{C}ech compactification $\beta{\mathbb{N}}$ for discrete space $\mathbb{N}$ is homeomorphic to the subspace of all ultrafilters on $\mathbb{N}$, so the question of how all topologies on a fixed infinite set $X$ behave collectively as a natural subspace of $\px$ is interesting and yet is little explored.  We establish that complete sublattices of $\PP(\PP(X))$ provide a rich supply of compact subsets within $\px$. It is then possible to find infinite compact subsets of $Top(X)$ by purely  lattice-theoretic means and to gain further insight into the topological complexity of $Top(X)$.

The first section of this paper focuses on extending the aforementioned results by describing the equivalence between the topological closures of sublattices of $\px$ and their (completely distributive) completions. We exploit such an equivalence when searching for countably infinite compact subsets within $Top(X)$ and in crystalizing the Borel complexity of $Top(X)$. The last section is devoted to describing other infinite compact subsets of $Top(X)$ including, in particular, copies of the Stone-\v{C}ech and one-point compactifications of discrete spaces.

\section{Preliminaries}
For convenience and unless otherwise indicated, $\px$ shall denote the usual Boolean algebra for an infinite set $X$ equipped with the subset inclusion order ($\subseteq$) in addition to the usual product space, where $2$ is the discrete space. We reserve the use of the symbol $\subset$ for cases of proper or strict containment only. The topology of any subset $P$ of $\px$ then is simply the usual subspace topology on $P$ (which we shall denote where necessary by $P_\Pi$) while $\overline{P}$ will denote the topological closure of $P$ in the space $\px$. \\
For a sublattice $P$ of $\px$, we denote by $\hat{P}$ its lattice-theoretic completion. Thus $\hat{P} = \bigcap \{L \subseteq \px \mid P \subseteq L$, $L$ a sublattice of $\px$ and $L = \hat{L}\}$. Of course, finite sublattices are trivially complete and, in general, the \textit{free completion} of a lattice does not exist (see \cite{MR0108447}, \cite{MR0003614}, \cite{MR0006143}).  That said, the \textit{completely distributive completion} of any partial order exists and is unique up to isomorphism \cite{MR2163422}. Moreover, for a partial order $P$:
 $$x \in \hat{P} \mbox{ if and only if } x = \bigwedge \bigvee S \mbox{ for all } S \subseteq P \mbox{ so that }x \leq \bigvee S.$$

Since $\px$ (as well as any sublattice thereof) is completely distributive we have an explicit description of each element in the lattice-theoretic completion of any sublattice of $\px$.

\begin{definition} Let $(P, \leq)$ be a poset with $p \in P$. We define $p^{\downarrow} = \{ x \in P \mid x \leq p\}$, $p^{\uparrow}= \{ x \in P \mid p \leq x\}$, $p_{\downarrow} =p^{\downarrow}\smallsetminus \{p\}$ and $p_{\uparrow}= p^{\uparrow}\smallsetminus \{p\}$.

\end{definition}
\begin{definition} We shall adopt the following notation:
\begin{itemize}
\item[(i)]If $P$ is a lattice, and $S \subseteq P$, then we denote by $<S>_L$ the sublattice of $P$ generated by closing off $S$  under finite meets and finite joins.
\item[(ii)] If $\mathcal{S} \subseteq \PP(X)$, then $<\mathcal{S}>_T$ denotes the topology $\sigma$ on $X$ generated by closing off $\mathcal S$ under finite intersections and arbitrary unions. Thus $\mathcal{S} \cup \{\emptyset, X\}$ is a subbase for $\sigma$.
\end{itemize}
\end{definition}
\begin{definition} Given any $S \subseteq 2^{\PP(X)}$ we let $R_S = \{ a \in 2^{\PP(X)} \mid \forall b \in S,$ either $b \subseteq a$ or $a \subseteq b\}$ and refer to it as the \textit{set of relations of} $S$.

\end{definition}

\section{Completeness and compactness of sublattices}

\begin{lemma} Let $P$ be a sublattice of $\px$, let $S \subseteq P$ and let $x = \bigvee S$. If $x \not \in P$, then $x$ is a limit point of $P$.
\end{lemma}
\begin{proof}
Let $\bigcap A_i^+ \cap \bigcap B^-_j$ be an arbitrary open neighbourhood of $x$. Then for each of the finitely many $i$, there is $ s_i \in S$ such that $A_i \in s_i$; furthermore $B_j \not \in s_i$ for each $j$ and each $i$. Thus $\bigvee_i s_i \in \bigcap A_i^+ \cap \bigcap B^-_j \cap P$ and clearly $\bigvee_i s_i \not = x$.

\end{proof}

Recall that the \textit{interval topology} on a poset $P$ is the one generated by $\{x^{\upar} \mid x \in P\} \cup \{x^{\dona} \mid x\in P\} \cup \{P, \emptyset\}$ as a subbase for the closed sets; we denote it by $P_<$. The \textit{order topology} $P_O$ on a lattice $P$ is defined in terms of Moore-Smith convergence. A filter $\F$  of subsets from $P$ is said to Moore-Smith-converge to a point $l \in P$ whenever

$$\bigwedge_{F \in \F} \bigvee F = l = \bigvee_{F \in \F} \bigwedge F.$$

We then take $F \subseteq P$ to be closed if and only if any convergent filter that contains $F$ converges to a point in $F$. For a lattice $P$, $P_< \subseteq P_O$ \cite{MR0006496}.

\begin{lemma} \label{lem:threetop} Let $P$ be a sublattice of $\px$. Then $P_< \subseteq P_{\Pi} \subseteq P_O$ and all three topologies coincide when $P$ is a complete sublattice of $\px$. Moreover, all three topologies on $P$ are compact if and only if $P$ is complete.

\end{lemma}

\begin{proof} The first inequality is true since for any $x \in P$, we have that

$$x^{\upar} \cap P = \bigcap_{A \in x} (A^+ \cap P).$$

A similar argument holds for $x^{\dona} \cap P$. For the second inequality and without loss of generality, take any subbasic closed set $A^+$ and let $\F$ be a convergent filter in $P$ containing $A^+ \cap P$; this forces $\bigwedge (A^+ \cap P)$ to exist in $P$. Since $P$ is a sublattice of $\px$ then

$$\bigvee_{F \in \F} \bigwedge F = \bigcup_{F \in \F} \bigcap F$$

and $A \in \bigcup\limits_{F \in \F} \bigcap F$. Consequently, $ \bigcup\limits_{F \in \F} \bigcap F \in A^+ \cap P$.

Next, if $P$ is complete
then for any $A \subseteq X$, it follows that $\bigcap (A^+ \cap P) \in P$ and $\bigcup (A^- \cap P) \in P$. In turn, $(A^+ \cap P) = (\bigcap (A^+ \cap P))^{\upar}$, $(A^- \cap P) = (\bigcup (A^- \cap P))^{\dona}$ and $P_< =  P_{\Pi}$.
 Since $P_<$  is $T_2$ then $P_O$ must be compact Hausdorff \cite{MR0167439} and $P_O = P_<$. The last claim is true since Frink \cite{MR0006496} shows that a complete lattice is compact in its interval topology if and only if it is complete.

\end{proof}

Not only is a sublattice $P$ compact in $\px$ precisely when $P$ is complete but also the closure of $P$ within $\px$ is indeed its lattice theoretic completion:

\begin{theorem} \label{thm:limpoint} Given an infinite sublattice $P$ of $\px$ and $x \in \px$,
 \begin{itemize}

      \item[(i)]$x$ is a limit point of $P$ only if $x$ can be expressed in the form $\bigwedge_{j \in J} \bigvee_{i \in I_j} x_{i,j}$, where $x \not = x_{i,j}$ for each $ i,j$  and $\{x_{i,j} \mid i \in I_j\}_{j \in J}$ are infinite subsets of $P$.
          \item[(ii)] ${\overline{P}}= \hat{P}$; that is, the topological closure of $P$ in $\px$ coincides with its lattice-theoretic completion.

    \end{itemize}

\end{theorem}

\begin{proof}
(i) Let $x$ be a limit point of $P$ and assume, without loss of generality, that $x$ is infinite. If $x = \PP(X)$ then since all neighbourhoods of $x$  meet $P$, we must have $x = \bigvee (x_{\dona} \cap P)$ and we are done. Otherwise fix any $B \not \in x$ and notice that $\forall A \in x$ we have that $A^+ \cap B^- \cap P \not = \emptyset$. Furthermore $|A^+ \cap B^- \cap P| \geq \aleph_0$ since otherwise  we can find (in Hausdorff $\px$) a neighbourhood of $x$ which is disjoint from $P$.  Holding $B$ fixed, it is simple to see that $x \subseteq \bigvee\limits_{A\in x}\left( \bigvee (A^+ \cap B^- \cap P)\right)$ while $B \not \in \bigvee\limits_{A\in x}\left( \bigvee (A^+ \cap B^- \cap P)\right)$. To this end, we must only intersect all such suprema for each $B \not \in x$ and we have the required form. In symbols:

$$x = \bigwedge_{B \not \in x} \bigvee_{A \in x}\left(\bigvee (A^+ \cap B^- \cap P)\right).$$

For (ii) we must only notice that (i) $\Rightarrow$ (ii). Indeed, take any $x \in \hat{P}$. If $x \in P$, we are done. Otherwise, if $x = \bigvee S$, for $S \subseteq P$ then Lemma 3.1 applies. The last possibility is for $x = \bigwedge\limits_{k \in K} \bigvee S_{k}$ where $S_{k} \subseteq P$ and $x \subset \bigvee S_{k}$. Take a basic open set $\bigcap A_i^+ \cap \bigcap B^-_j$ about $x$ and observe that for all $i$ and for all $k \in K$, $A^+_i \cap S_{k} \not = \emptyset$.  As for the $B_j^-$, we know that for any $j$ we can find a $k_j$ for which $B_j \not \in \bigvee S_{k_j}$. Hence, for each $j$ take a finite collection of elements from its corresponding $<S_{k_j}>_L$ (i.e. the one for which $B_j \not \in \bigvee S_{k_j}$) so that the join of such a collection contains all $A_i$. Taking the intersection of all such collections for each $j$ we have an element of $P$ that is contained in the aforementioned basic open set and thus $x$ is a limit point of $P$.

%\item[(i)] In view of $(i)$ we have that $\overline{P}^\Pi \subseteq \overline{P}$. Seeking a contradiction, assume that $\overline{P}^\Pi$ is strictly contained in %$\overline{P}$.  If $\overline{P}^\Pi$ is a lattice, then it cannot be complete and so we can find an infinite collection of points in $\overline{P}^\Pi$ whose  join %is not in $\overline{P}^\Pi$ (What about meet?). However, every neighbourhood of the join meets $\overline{P}^\Pi$, a contradiction. Thus $\overline{P}^\Pi$ is not a lattice and also cannot be complete, that is, cannot be closed under infinite joins and meets, for suppose otherwise and that $\overline{P}^\Pi$ is not a lattice and yet is complete (thus implicit that $\overline{P}^\Pi$ is infinite). Then, take two elements in $\overline{P}^\Pi$ whose union (join) does not exist in $\overline{P}^\Pi$. Take all of the sups of elements containing the two points under consideration. Then the meet of all of those new found elements is the union of the aforementioned pair of elements from TP (we can take meets and joins since TP is complete). Thus there must exist an infinite collection $\{x_i \mid i \in I\} \subseteq \overline{P}^\Pi$ such that $ x = \bigvee \{x_i \mid i \in I\} \not \in \overline{P}^\Pi$.  Now the lattice $<\{x_i \mid i \in I\}>_{L}$ is contained in $\overline{P}^\Pi$ and any open neighbhourhood of $x$ must intersect it non-trivially. Hence $x \in  \overline{P}^\Pi$  and we have a contradiction.

\end{proof}

Thus for example,  a chain $\Omega$ in $\px$ is by default a sublattice of $\px$ and so its closure $\overline{\Omega}$ in $\px$ is its lattice-theoretic completion, which is again a chain. In fact, observe that $$\hat{\Omega} (= \overline{\Omega}) = \left\{\bigcap \Omega \right\} \cup \left\{\bigcup \Omega \right\} \cup \Omega \cup \left\{ \bigcup (b_{\downarrow} \cap \Omega) \mid b \in R_\Omega \right\} \cup \left\{ \bigcap (b_{\uparrow} \cap \Omega) \mid b \in R_\Omega \right\}.$$

 \begin{remark} Let $(X, \sigma)$ be any topological space containing a convergent sequence $(x_n)_{n\in\omega}$ where $x_n \rightarrow x_{\omega}$. That $x_n\rightarrow x_{\omega}$ is equivalent to demanding that any open set containing $x_{\omega}$ must contain all but finitely many points from $(x_n)$. Notice that the same is true for $\omega$ in the ordinal space $\omega + 1$ (with the order topology) . Moreover, any natural number in $\omega +1$ is isolated and hence $\omega$ is a discrete subspace of $\omega + 1$. Thus $\omega+1$, as an indexing set for any convergent sequence with its limit $\{x_n, x_{\omega}\}$, sets up a natural and continuous mapping $\phi: \omega + 1 \rightarrow \{x_n, x_{\omega}\}_{n \in \omega}$ (where $n \rightarrow x_n$) whereupon compactness naturally transfers. With that in mind, let $\Omega = \{a_1, a_2, \ldots\}$ be a well-ordered chain in $\px$ with $\alpha$ as its indexing ordinal.
  If $\beta \in \alpha$ is a limit ordinal, then any open set about $\beta$ contains infinitely many ordinals below $\beta$. Notice that this might not be the case with $a_{\beta}$, for if $a_{\beta} \not = \bigcup (a_{\beta})_{\dona}$ then $a_{\beta}$ can be separated from $(a_{\beta})_{\dona}$ by means of open sets. Thus, the bijection $h:\alpha \rightarrow \Omega$ for which $\beta \mapsto a_{\beta}$ is clearly open: $h$ is a homeomorphism if and only if for any limit ordinal $\beta \in \alpha$ we have $a_{\beta} = \bigcup (a_{\beta})_{\downarrow}$.

\end{remark}

\section{For $X$ infinite, $Top(X)$ is neither a $G_\delta$ nor an $F_\sigma$ set}

\begin{lemma} \label{lem:interopen} For any $\{A_i \mid i \in \omega \} \subseteq \mathcal{P}(X)$, $\bigcap_{i \in \omega} A_i^+ $ contains a sublattice of $\mathcal{P}(X)$ that is not join complete; that is, $\left(\bigcap_{i \in \omega} A_i^+ \right) \cap \left( LatB(X) \smallsetminus Top(X) \right) \not = \emptyset$.

\end{lemma}

\begin{proof} Consider $<\{A_{i} \mid i \in \omega \}>_L$ and suppose that it is join complete (otherwise, we are done). Notice that its countable cardinality demands that only finitely many of the $A_{i}$s can be singletons. Since $X$ is infinite, we may choose a countable infinite collection of singletons $\Set = \{\{p\} \mid p \in X\}$ from $\PP(X)$ and generate a lattice $K= <\{A_{i}\} \cup \Set>_L$. Then $K$ cannot be join complete for there are uncountably many subsets of $\cup \Set$ (i.e. joins of $\Set$) and only $\aleph_0$ many elements in $K$.

\end{proof}

\begin{theorem} $Top(X)$ is not a $G_{\delta}$ set.
\end{theorem}

\begin{proof} Suppose that $Top(X) = \bigcap_{k \in \omega} \mathcal{O}_k$, where

$$\mathcal{O}_k = \bigcup_{\alpha \in \beta_k} \left((\bigcap_{i_{\alpha} \leq n_{\alpha}} A^+_{i_{\alpha}}) \cap (\bigcap_{j_{\alpha} \leq m_{\alpha}} B^-_{j_{\alpha}})\right).$$

Now, the discrete topology $\mathcal{D}$ on $X$ must be in this intersection of open sets. Thus for each $k \in \omega$, it must belong to at least one basic open set of the form $(\bigcap_{i_{\alpha} \leq n_{\alpha}} A^+_{i_{\alpha}}) \cap (\bigcap_{j_{\alpha} \leq m_{\alpha}} B^-_{j_{\alpha}})$ and since $\mathcal{D}$ contains all sets, then no subbasic open set can be of the form $B^-$. That is, $\mathcal{D}\in \bigcap_{k \in \omega} A^+_{k}$ after some renumeration of the $A$s. Applying Lemma~\ref{lem:interopen} to $\bigcap_{k \in \omega} A^+_{k}$
we can find a sublattice of $\mathcal{P}(X)$ that belongs to $\bigcap_{k \in \omega} A^+_{k}$ and that is not join complete - a contradiction.

\end{proof}

 In fact, Lemma~\ref{lem:interopen} proves something much stronger. Define recursively:

 \begin{align*}
 G^0_{\delta} &:= \{ \mbox{all } G_{\delta} \mbox{ sets}\} \\
 G_{\delta \sigma}^0 &:= \{ \mbox{all countable unions of } G_{\delta} \mbox{ sets} \}\\
 G_{\delta}^\beta & := \{ \mbox{all countable intersections of } G_{\delta \sigma}^{\beta-1} \mbox{ sets}\} & \mbox{(for $\beta$ a successor ordinal)}\\
 G_{\delta \sigma}^\beta & := \{ \mbox{all countable unions of } G_{\delta}^{\beta} \mbox{ sets}\} & \mbox{(for $\beta$ a successor ordinal)}\\
 G_{\delta}^\gamma &:= \bigcup_{\beta \in \gamma} G_{\delta}^\beta & \mbox{(for $\gamma$ limit ordinal)}\\
G_{\delta \sigma}^\gamma &:= \bigcup_{\beta \in \gamma} G_{\delta \sigma}^\beta & \mbox{(for $\gamma$ limit ordinal)}
 \end{align*}

 %Notice that $G_{\delta}^\beta \subset G_{\delta \sigma}^\beta \subset G_{\delta}^{\beta + 1}$ and that for any limit ordinal, $\gamma$, $G_{\delta}^\gamma = G_{\delta \sigma}^\gamma$. The previous theorem tells us that any countable intersection of open sets about the discrete topology on $X$ contains a lattice that is not join complete. In turn, $Top(X)$ is not $G_{\delta \sigma}^0$. But that is not all; any $G_{\delta}^1$ is of the form $ \bigcap_{i \in \omega} \left( \bigcup_{j \in \omega} \left( \bigcap_{k \in \omega} O_{i,j,k} \right) \right)$, where each $O_{i,j,k}$ is an open set in $\px$. If $Top(X) \in G_{\delta}^1$ then it must be possible to express $Top(X)$ as done above. Since $\D(X) \in Top(X)$ then for each $j \in \omega$ there must exist a $G_\delta$ set that contains $\D(X)$.  Collecting all such $G_\delta$  sets (for each $i \in \omega$) leaves us with a countable collection of basic open sets of the form $\bigcap A_i^+$, for $A_i \in \PP(X)$. Just as it was done in Theorem~\ref{thm:notgd} we can concoct a lattice that is not join complete and satisfies the aforementioned collection of basic open sets and, thus, $Top(X) \not \in G_{\delta}^1$. A simple inductive argument tells us that the same is true for any $n \in \N$.

%\begin{corollary} $Top(X) \not \in G_{\delta}^\omega$.
 %\end{corollary}

 Take for any $n \in \N$ a set $G_n$, say, from $G_{\delta}^n$ and assume that $Top(X) = \bigcap_{n \in \omega}  G_n$. Since $\D \in  \bigcap_n G_n$ then for each $n \in \N$ we can find a countable collection $\{A_{i_n}\}_{i \in \omega}$ of subsets of $X$ corresponding to each $G_n$ so that

 $$\D \in \bigcap_{i_n} A^+_{i_n} \subset G_n$$

 and consequently

 $$\D \in \bigcap_{n \in \omega} \bigcap_{i_n} A^+_{i_n} \subset \bigcap_{n \in \omega} G_n.$$

By Lemma~\ref{lem:interopen}, there is a lattice that is not join complete belonging to $ \bigcap_{n \in \omega} \bigcap_{i_n} A^+_{i_n}$. Hence, $Top(X) \not = \bigcap_{n \in \omega}  G_n$. Notice that the same is true for any countable limit ordinal. That is, for any $\beta \in \omega_1$ so that $\D \in G \in G_\delta^\beta$ it is possible to extract a countable collection of open sets $A_i$ ($i \in \omega$) so that $\D \in \bigcap_{i \in \omega} A^+_i \subseteq G$, in which case Lemma~\ref{lem:interopen} completes the proof.

  \begin{corollary} $Top(X) \not \in G_{\delta}^\beta$ for $\beta \in \omega_1$.
 \end{corollary}

In other words, it is not possible to generate (in the sense of Borel sets) $Top(X)$ by means of open sets. If $\px$ was metrizable (which it is not) then the above corollary would suffice to show that $Top(X)$ is not a Borel set.

\begin{corollary} $Top(X)$ is not \v{C}ech complete.
\end{corollary}

\begin{proof} We showed above that any countable intersection of open sets from $2^{\PP(X)}$ containing the discrete topology on $X$ contains an element of $LatB(X)\smallsetminus Top(X)$. Hence, $Top(X)$ is not a $G_{\delta}$ set in $LatB(X)$.

\end{proof}

\begin{theorem} $Top(X)$ is not an $F_{\sigma}$ set.

\end{theorem}

\begin{proof} If $Top(X)$ is an $F_{\sigma}$ set, then it must be of the form

$$ Top(X) = \bigcup_{k \in \omega} \C_k$$
where each $\C_k$ is a closed set. We will show by contradiction that at least one such closed set must contain a sequence of topologies whose limit is not a topology. Since the limit of any sequence must be present in the closure of the sequence, then the aforementioned closed set will contain an element that is not a topology. We prove the above for $|X| = \aleph_0$ and note that the same is true for any $X$ with $|X| \geq \aleph_0$.

Let $k : [0,1] \rightarrow \PP(\N)$ be an injective order morphism so that $\forall a \in [0,1]$, $\bigcup_{b < a} k(b) = k(a)$ and $k(1) \not = \N$. That is, $k([0,1])$ is a dense and uncountable linear order in $\PP(\N)$ where $a < b \Rightarrow k(a) \subset k(b)$. Next, for any $a$ define $\tau_a = \PP(k(a)) \cup \{\N\}$. Notice that $\forall a \in [0,1]$, $\tau_a \in Top(\N)$ and $\bigcup_{b < a} \tau_b \not \in Top(\N)$ (since $k(a) \not \in \bigcup_{b < a} \tau_b$),  and $\{\tau_a\}_{a \in [0,1]}$ is an uncountable dense linear order in $Top(\N)$. If $Top(\N)= \bigcup_{k \in \omega} \C_k$ where each $\C_k$ is closed then there must exist one set $\C$ from $\{\C_k\}_{k \in \omega}$ which contains an uncountable set $D \subset \{\tau_a\}_{a \in [0,1]}$ for which $\mu(D) > 0$ (non-zero measure). We immediately get that $D$ must contain a densely ordered subset that in turn contains a strictly increasing sequence, call it $S$. By Theorem~\ref{thm:limpoint}, $ \bigcup S \in \hat{S}$ but $\bigcup S \not \in Top(\N)$. To this end we have $\bigcup S \in \C$, a contradiction.

\end{proof}

\begin{corollary} For $\beta \in \omega_1$, the following are equivalent:

\begin{enumerate}[(a)]
\item $Top(X) \in G_{\delta}^\beta$.
\item $Top(X)$ is a  $G_{\delta}$ set.
\item $Top(X)$ is an $F_{\sigma}$ set.
\item $Top(X)$ is  \v{C}ech complete.
\item $X$ is finite.

\end{enumerate}

\end{corollary}

\section{Compact infinite subsets of $Top(X)$}

In this section, we provide examples of compact infinite subsets of $Top(X)$. Note in particular that any countable chain of topologies must converge to its union which may not itself be a topology. For example, consider the nested sequence of finite topologies $\{\tau_i \mid \tau_i = \PP(\{x_0, x_1. \ldots, x_i\}) \cup \{X, \emptyset\}\}$, where $\{x_0, x_1, \ldots\}$ is a countable infinite subset of $X$. Then $\tau_n \rightarrow \bigcup \tau_i$ but notice that $\bigcup \tau_i$ fails to be a topology as $\{x_0, x_1, \ldots\}$ does not belong to any $\tau_i$. Of course, $LatB(X)$, as a compactification of $Top(X)$, will contain all such limits.

\begin{example} For simplicity, take a countable infinite subset $C$ of $X$. Enumerate $C = \{a_0, a_1, a_2, \ldots \}$ and create a sequence  in $\PP(C)$ as follows: $C_0 = \{a_0\}$, $C_1 = \{a_0, a_1\}$, ..., $C_{\omega} = C$. Now, for any $n \in \omega$ let $\tau_n = \{C_m \mid m\leq n\} \cup \{X, \emptyset\}$ and $\tau_{\omega} = \{C_m \mid m \in \omega\} \cup \{C, X, \emptyset\}$; then $(\tau_n)_{n \in \omega}$ converges to (non-Hausdorff) $\tau_{\omega}$ in $Top(X)$. Indeed, let $B = \bigcap A^+_i \cap \bigcap B^-_j$ be a basic open set containing $\tau_{\omega}$. Then no $B_j = C_m$ for any $m \in \omega$ and any $A_i$ must be either $\emptyset$, $X$, $C$ or a $C_n$, for some $n \in \omega$. Since there are only finitely many $A_i$ then there exists an $m \in \omega$ for which all $A_i \in \tau_m$.

\end{example}

In view of the above example,  we can construct a convergent sequence of compact non-Hausdorff topologies, whose limit is both compact and Hausdorff.
\begin{example} Let $[a,b] \subset \R$, and define any strictly increasing sequence $\{x_n \mid a < x_n < b\}_{n \in \omega}$ whose limit is $b$. Next, let $\NN_b = \{(c,b] \mid a \leq c \leq b\}$, $\NN_a(x) = \{ [a,c) \mid c \leq x\}$ and

\begin{align*}
\tau_0 &=  <\{[a,b)\} \cup \NN_b \cup \NN_a(x_0)  \cup \{\emptyset\}>_T\\
\tau_1 &=  <\{[a,b)\}\cup \NN_b \cup \NN_a(x_1) \cup \{\emptyset\}>_T\\
\vdots \\
\tau_{\omega} &= \bigcup_{i \in \omega} \tau_i
\end{align*}

By design, $\tau_n\ \rightarrow \tau_{\omega}$. Observe also that $\tau_{\omega}$ is the usual Euclidean topology on $[a,b]$ and so is compact and Hausdorff. Indeed, given any $c \in [a,b]$ we must only check that $(c,b] \in \tau_{\omega}$. Since $x_n\rightarrow b$ then there exists a $k \in \omega$ so that $x_k > c$, hence $(c,b] \in \tau_k$. To show that each $\tau_n$ ($n \in \omega$) is compact, take an open cover $\mathcal{C}$ of $[a,b]$ from $\tau_k$ ($k\in \omega$). Notice that since $a$ must be covered then $[a,c_1) \in \C$ for some $c_1 \leq x_k$. Since $x_k$ must also be covered by some element in $\C$ then, for some $c_2 \leq x_k$, either $(c_2,b)$ or $(c_2,b]$ belong to $\C$. If $c_1 > c_2$ then we're done. Otherwise, notice that $\tau_k \upharpoonright [c_1,c_2]$ is the usual topology on $\R$ restricted to $[c_1,c_2]$ (which yields a compact space). Finally, no $\tau_n$ is Hausdorff (since $b$ can't be separated from all points in $[a,b]$).
\end{example}

The above example confirms that the collection of compact non-$T_2$ topologies on a set $X$ fails to be closed given the existence of a countable chain of compact non-$T_2$ topologies whose union (and topological limit)is a $T_2$ and compact topology.

Even though any compact topology is contained in a maximal compact topology \cite{kovar:DSP:2005:118}
it is possible to construct strictly increasing sequences of compact topologies whose limits are not compact. Consider the half-open half-closed interval $[a,b)$ equipped with the (convergent) sequence of topologies $\tau_i$ as in the previous Example, with $\NN_b = \{(c,b) \mid a \leq c\}$ modified accordingly. It is clear that every topology, with the exception of $\tau_{\omega}$, is compact.

Nested sequences are not the only type of compact infinite subsets of $Top(X)$. Recall that an \emph{atom} in $Top(X)$ is a topology of the form $\{\emptyset, A, X\}$ where $A$ is a nonempty and proper subset of $X$. Consider then the following theorem where $\II$ denotes the trivial topology on $X$.

\begin{theorem} Let $\TT$ be an infinite collection of atoms in $Top(X)$. Then $\overline{\TT} = \TT \cup \{\II\} $.

\end{theorem}

\begin{proof} Let $\A \subset \PP(X)$ so that $\emptyset \not \in \A$, $\tau_A = \{X, \emptyset, A\}$ and $T_{\A} = \{\tau_A \mid A \in \A\}$. Consider the following closed set

$$
\C = X^+ \cap \emptyset^+ \cap \left(\bigcap_{D \in \PP(X) \setminus \A} D^- \right) \cap \left(\bigcap_{B, C \in \A} (B^- \cup C^-) \right)
$$\\

where, of course, $B \not = C$. Then $T_\A \subseteq \C$ and $\II \in \C$. Any family that contains any element from $\PP(X) \setminus \A$ can't belong to $\C$ and any family (and topology) that contains elements from $\A$ can contain at most one. Thus $T_\A \cup\{\II\} = \C$. Finally, any neighbourhood of $\II$ must intersect $T_{\A}$ and the result follows.

\end{proof}

\begin{corollary} Let $\TT$ be any infinite collection of atoms in $Top(X)$. Then $\TT \cup \{\II\}$ is the one-point compactification of $\TT$ in $2^{\PP(X)}$.

\end{corollary}

%\begin{theorem} \label{thm:genfam} Let $\A \subseteq 2^{\PP(X)}$ so that $\bigcap \A \not = \emptyset$ and for any infinite $\A' \subseteq \A$, $\bigcap \A' = \bigcap %\A$. Then $\A \cup \{\bigcap \A\} = \overline{\A}$ is the one-point compactification of $\A$.

%\end{theorem}

Given two topologies in $Top(X)$ we say that they are disjoint provided their intersection is the trivial topology $\II$ on $X$.

\begin{theorem} \label{thm:fortop}
Let $\TT$ be an infinite collection of pairwise disjoint topologies on $X$. Then $\TT \cup \{\II\} = \overline{\TT}$ is the one-point compactification of $\TT$ in $2^{\PP(X)}$.
\end{theorem}

\begin{proof} In the following expression, $\sigma$ and $\rho$ denote topologies in $\TT$ while $A$ and  $B$ denote certain nonempty and proper open subsets of $X$. We claim that

$$
\C = X^+ \cap \emptyset^+ \cap \left(\bigcap_{D \not \in \cup \TT} D^- \right) \cap \left(\bigcap_{ \substack{ (A,B) \in (\sigma, \rho)  \\  \sigma \not = \rho }} (A^- \cup B^- )\right) \cap \left( \bigcap_{\substack{A \not = B \\ A, B \in \sigma }} (A^+ \cup B^-)\right)
= \TT \cup \{\II\}.
$$

Clearly $\C$ is closed and $\TT \cup \{\II\} \subseteq \C$. Let $x \in \px \smallsetminus \TT$ such that $\II \subset x$. If $x \not \subseteq \cup \TT$ then there must exist $O \in x$ so that $O \not \in \cup \TT$; since $x \not \in O^-$ then $x \not \in \C$. If $x \subseteq \cup \TT$ then it is either contained in a topology from $\TT$ or not. In the former case, take $\rho \in \TT$ where $x \subset \rho$, $U \in \rho \setminus x$ and $V \in \rho \cap x$ where $V$ is neither empty nor $X$. Then $x$ fails to belong to $U^+ \cup V^-$. Otherwise we are guaranteed a pair of distinct topologies, $\rho$ and $\sigma$, in $\TT$ for which there exists $V \in x \cap \rho$ and $U \in x \cap \sigma$ and neither open set is equal to $X$ or $\emptyset$. Hence, $x \not \in V^- \cup U^-$ and we have that $\C = \TT \cup \{\II\}$. Finally, any neighbourhood of $\II$ must have a cofinite intersection with $\TT$ and the result follows.

\end{proof}

\begin{corollary} For any (infinite) discrete space $Y$, where $|Y| \leq 2^{|X|}$, $Top(X)$ contains a copy of its one-point compactification.
\end{corollary}

\begin{proof} There are $2^{|X|}$ atoms in $Top(X)$ and all are disjoint from each other.

\end{proof}

\section{$\beta \mathbb{N}$ in $Top(\mathbb{N})$}

An \emph{ultratopology} $\TT$ on a set $X$ is of the form $\TT = \PP(X \smallsetminus \{x\}) \cup \U$, where $\U$ is an ultrafilter on $X$ and $\{x\} \not \in \U$. We shall use the notation $\TT_\U$ for such an ultratopology when we wish to identify the associated ultrafilter $\U$. We denote by $Ult(X)$ the set of all ultratopologies on $X$. Whenever $\U$ is a principal (non-principal) ultrafilter, $\TT$ is called a principal (non-principal) ultratopology. Denote by TYPE($x$) the set $  \{ \F \mid \F$ is an ultrafilter and $\{x\} \not \in \F\}$ and by TYPE[$x$] the set $\{ \TT \in Top(X) \mid \TT$ is an ultratopology and $\{x\} \not \in \TT\}$. Note that $\{$TYPE[$x$]$: x \in X\}$ is a partition of $Ult(X)$.

Given $X$, define $\U_X$ to be the set of all ultrafilters on $X$ and for any $\F \in \U_X$ and $A \subset X$ let
$$\F_A = \F \upharpoonright (X \smallsetminus A) = \{ N \cap (X \smallsetminus A) \mid N \in \F \}.$$

In other words, $\F_A$ is the \textit{trace} of $\F$ on $X \smallsetminus A$. Whenever $A = \{a\}$ we let $\F_{\{a\}} = \F_a$.

\begin{lemma} \label{lem:3part} Let $n \in \N$ and denote $\N' = \N  \smallsetminus \{n\}$ then

\begin{enumerate}[(i)]
\item $\forall \F \in$ TYPE($n$), $\F = \F_n \cup \{M \cup \{n\} \mid M \in \F_n\}$,
\item $\forall \F \in$ TYPE($n$), $\F_n \in \U_{\N'}$ and
\item $\F, \G \in$ TYPE($n$) so that $\G \not = \F$ implies that $\F_n \not = \G_n$.
\end{enumerate}
\end{lemma}

\begin{proof}

\begin{enumerate}[(i)]
\item We must only show that $\F_n \subset \F$. Indeed, if $\F_n \subset \F$ then given any $A \in \F_n$ we have that $A \cup \{n\} \in \F$ since $\F$ is a filter. So let $A \in \F_n$ and notice that either $A \in \F$ or $A \cup \{n\} \in \F$. Since the former case is trivial assume that $A \cup \{n\} \in \F$. Notice that since $\F \in$ TYPE($n$) then $(\N' \smallsetminus A) \cup \{n\} \not \in \F$ and thus $A =\N \smallsetminus ((\N' \smallsetminus A) \cup \{n\}) \in \F$.

\item Take $A \in \F_n$. If $A \subset B \subseteq \N'$, then $B \in \F$ and consequently $B \in \F_n$. Next, let $A,B \in \F_n$ and notice that $A,B \in F$ and $A \cap B \in \F$. Thus, $A \cap B \in \F_n$. Lastly, let $A \subset \N'$ and notice that either $A$ or its complement in $\N$ belong to $\F$. In the former case we are done so assume $\N \smallsetminus A \in \F$. To this end, we must only notice that $\N' \smallsetminus A = (\N \smallsetminus A) \cap \N' \in \F_n$.
\item This follows directly from (i).
\end{enumerate}

\end{proof}

\begin{theorem} For any $n \in \N$ the mapping $\mathfrak{F}:$ TYPE($n$)$\rightarrow \U_{\N'}$ for which $\F \mapsto \F_n$ is a bijection.

\end{theorem}

\begin{proof} Lemma~\ref{lem:3part} part (ii) tells us that the range is well-defined. Also, for any filter $\HH \in \U_{\N'}$ it is easy to see that $\F_{\HH} = \HH \cup \{M \cup \{n\} \mid M \in \HH\} \in$ TYPE($n$) and that $\F_{\HH} \upharpoonright \N' = \HH$. Lastly, from part (iii) of Lemma~\ref{lem:3part} we get injectivity.

\end{proof}

\begin{theorem} \label{thm:homeo} For any $n \in \N$, TYPE[$n$] is homeomorphic to $\beta \N$.

\end{theorem}

\begin{proof} Since TYPE($n$) can be bijected with $\U_{\N'}$ (by $\F \mapsto \F_n$) then the same is true of TYPE[$n$]. That is, for any $\TT_{\F} \in$ TYPE[$n$] we canonically map $\TT_{\F} \mapsto \F_n$ so that $\TT_{\F} = \PP(\N')  \cup \F$. Recall that a subbase for $\beta \N'$ is comprised of sets of the form $A' = \{\F \in \U_{\N'} \mid A \in \F\}$ for all $A \in \PP(\N')$. We claim that for any $A \subseteq \N'$, $A' \mapsto (A \cup \{n\})^+ \cap$ TYPE[$n$]. Indeed, if $A \in \F_n$ for some $\F \in$ TYPE($n$) then by Lemma~\ref{lem:3part} $A \cup \{n\} \in \F$ and $\TT_{\F} \in (A \cup \{n\})^+ \cap $ TYPE[$n$].

Similarly,  for any $A \in \PP(\N')$, $A^+ \cap$ TYPE[$n$]$ = $ TYPE[$n$] which bijects to $\U_{\N'}$. If $A \subseteq \N$ with $n \in A$ then for any ultratopology $\TT_{\F}$ in $A^+ \cap$ TYPE[$n$] it must be the case that $A \smallsetminus \{n\} \in \F_n$. Consequently, $A^+ \cap$ TYPE[$n$] $\mapsto (A \smallsetminus \{n\})'$.

\end{proof}

In a nutshell, we have the following diagram of the above claim:\\

\begin{displaymath}
\xymatrix{
&\beta\N' \ar[d] \ar@{.>}[drrrr]^{\F_n \mapsto \TT_{\F}}\\
& \U_{\N'} \ar[u]  \ar[rr]_{\F_n \mapsto \F}
&
&\mbox{TYPE}(n) \ar[rr]_{\F \mapsto \TT_{\F}} \ar[ll]
&
&\mbox{TYPE}[n] \ar[ll] \ar@{.>}[ullll]
}
\end{displaymath}\\

\begin{corollary} The $F_{\sigma}$ set $Ult(\N) = \bigcup_{n \in \N} \mbox{TYPE[}n\mbox{]}$ is not compact.

\end{corollary}

\begin{proof} Consider the following open cover of $Ult(\N)$:

$$\bigcup_{n \in \N} \{n\}^-.$$

Note that $\forall n \in \N$, $\{n\}^- \cap Ult(\N) = $ TYPE[$n$] and so no finite subcollection of the above cover can cover $Ult(\N)$.

\end{proof}

In particular, the discrete topology on $\N$ is a limit point of $ Ult(\N) $. Lemma~\ref{lem:3part} and Theorem~\ref{thm:homeo} can be extended to any infinite set $X$. That said, for $|X| > \aleph_0$, $Ult(X)$ is not an $F_{\sigma}$ set.

\begin{theorem} For $Y$ a discrete space with $|Y| \leq |X|$, $Top(X)$ contains a copy of $\beta Y$.
\end{theorem}

\begin{proof}The proof is trivial for $|Y| = |X|$. Otherwise, take a copy of $\beta Y$ within $Top(Y)$ and an injection $i : Y \rightarrow X$. Then $\forall \rho \in Top(Y)$, $\rho_X = \{ A \subset X \mid i^{-1}(A) \in \rho\} \cup \{X\}$ is a topology on $X$. Moreover, $\{\rho_X \mid \rho \in \beta Y \}$ is a homeomorphic copy of $\beta Y$ in $Top(X)$.

\end{proof}

\bibliographystyle{abbrv}
\bibliography{mybib2}

\end{document}